\documentclass[a4paper,11pt]{article}%amsart}
\usepackage[utf8x]{inputenc}
\usepackage{a4wide}
\usepackage{amssymb}
\usepackage{amsmath}
\usepackage{amsthm}
\usepackage{latexsym}
\usepackage[mathscr]{euscript}
\usepackage{hyperref}
\usepackage[color]{changebar}
\usepackage{todonotes}

\theoremstyle{definition}
\newtheorem{definition}{Definition}[section]

\theoremstyle{plain}

\newtheorem{Proposition}[definition]{Proposition}

\newtheorem{Corollary}[definition]{Corollary}
\newtheorem{Definition}[definition]{Definition}

\newtheorem{Examples}[definition]{Examples}

\theoremstyle{remark}
\newtheorem{remark}[definition]{Remark}

\newcommand{\R}{\mathbb R}  
\newcommand{\N}{\mathbb N} 
\newcommand{\C}{\mathbb C} 
\newcommand{\Z}{\mathbb Z}

\newcommand{\eps}{\varepsilon}
\newcommand{\vphi}{\varphi}
\cbcolor{blue}

\newcommand{\D}{\mathcal{D}}
\newcommand{\cinfty}{C^\infty}

\newcommand{\oc}{\ensuremath{\mathcal{O}_C} }
\newcommand{\om}{\ensuremath{\mathcal{O}_M} }
\newcommand{\ilim}[1]{\displaystyle{\lim_{#1\rightarrow}}}
\newcommand{\plim}[1]{\displaystyle{\lim_{#1\leftarrow}}}
\newcommand{\hato}{\widehat\otimes}
\newcommand{\Sch}{\mathcal{S}}
\newcommand{\Linfty}{\mathrm{L}^\infty}
\newcommand{\B}{\mathcal{B}}
\newcommand{\Lp}{\mathrm{L}^p}
\newcommand{\Lq}{\mathrm{L}^q}
\newcommand{\Lo}{\mathrm{L}^1}
\newcommand{\Lt}{\mathrm{L}^2}
\newcommand{\pa}{\partial}
\newcommand{\cont}{\mathcal{C}}

\newcommand{\dlpmk}{(\D_{\Lp})_{-k}}
\newcommand{\dlp}{\D_{\Lp}}
\newcommand{\dlq}{\D_{\Lq}}
\newcommand{\dlo}{\D_{\Lo}}
\newcommand{\dli}{\D_{\Linfty}}

\newcommand{\lpmk}{(\ell^p)_{-k}}
\newcommand{\Four}{\mathcal{F}}
\newcommand{\mo}{^{-1}}

\newcommand{\limk}{(\ell^\infty)_{-k}}

\title{On the completeness of the space \oc}

\author{Michael Kunzinger\footnote{University of Vienna, Faculty of Mathematics, michael.kunzinger@univie.ac.at},
Norbert Ortner\footnote{University of Innsbruck, mathematik1@uibk.ac.at},
}

\begin{document}

\date{}

%\date{Received: date /Accepted: date}

\maketitle

\begin{abstract}
We explicitly prove the compact regularity of the $\mathcal{LF}$-space of double sequences $\ilim{k} (s\hato (\ell^p)_{k}) \cong \ilim{k}(s\hato (c_0)_{-k})$, $1\le p\le \infty$. As a consequence we obtain that the spaces of slowly and of uniformly slowly increasing $\cinfty$-functions $\om$ and $\oc$, respectively, are ultrabornological and complete.
Furthermore, we prove that $\ilim{k}(E_k\hato_\iota F) = (\ilim{k} E_k) \hato_\iota F$ if the inductive limit $\ilim{k}(E_k \hato_\iota F)$ is compactly regular.

\vskip 1em

\medskip

\medskip

\noindent 
\emph{Data Availability Statement:} Data sharing not applicable to this article as no datasets were generated or analysed during the current study.
\medskip

\noindent 
\emph{Conflict of Interest Statement:} There is no conflict of interest.
\end{abstract}

%%%%%%%%%%%%%%%%%%%%%%%%%%%%%%%%%%%%%%%%%%%%%%%%%%%%%%%%%%%%%%%%%%%%%%%%%%%%%%%%%%%%%%%%%%%%%%%%%%%%%%%%%%%%%%%%%%%%%%%%%%%%%%%%%%%%%%%%%%
%%%%%%%%%%%%%%%%%%%%%%%%%%%%%%%%%%%%%%%%%%%%%%%%%%%%%%%%%%%%%%%%%%%%%%%%%%%%%%%%%%%%%%%%%%%%%%%%%%%%%%%%%%%%%%%%%%%%%%%%%%%%%%%%%%%%%%%%%%
\section{Introduction}\label{sec:intro}
%%%%%%%%%%%%%%%%%%%%%%%%%%%%%%%%%%%%%%%%%%%%%%%%%%%%%%%%%%%%%%%%%%%%%%%%%%%%%%%%%%%%%%%%%%%%%%%%%%%%%%%%%%%%%%%%%%%%%%%%%%%%%%%%%%%%%%%%%%
%%%%%%%%%%%%%%%%%%%%%%%%%%%%%%%%%%%%%%%%%%%%%%%%%%%%%%%%%%%%%%%%%%%%%%%%%%%%%%%%%%%%%%%%%%%%%%%%%%%%%%%%%%%%%%%%%%%%%%%%%%%%%%%%%%%%%%%%%%
\subsection{Historical development. Properties of the spaces \texorpdfstring{$\oc, \om, \oc', \om'$}{OC,OM,OC',OM'}}
The spaces $\Sch = \Sch(\R^n)$ and $\Sch'=\Sch'(\R^n)$ of rapidly decreasing $\cinfty$-functions and of temperate distributions, respectively,
play a central role in L.\ Schwartz' theory of the Fourier transform. Temperate distributions on $\R^n$ can be continued to $\mathbb{S}^{n-1}$, which 
supplies the reason for calling them ``spherical distributions''. L.\ Schwartz also introduced the spaces $\oc'$ and $\om$ of convolutors and
multipliers, respectively, for the spaces $\Sch$ and $\Sch'$. For the notation used here and below, cf.\ Subsection \ref{subsec:notation}. 

The space $\oc'$ of rapidly decreasing distributions is defined as (cf.\ \cite{Schwartz}, p.\ 244)
\[
\oc' = \bigcap_{k=0}^\infty (\D'_{\Linfty})_k = \plim{k} (\D'_{\Linfty})_k.
\]
We have
\begin{equation}\label{eq:ocprimelimit}
\oc' = \plim{k} (\dot \B')_k = \plim{k} (\D'_{\Lp})_k \qquad (1\le p \le \infty).
\end{equation}
To see this, note first that the continuous inclusions 
\[
\plim{k}(\D'_{\Lp})_k \hookrightarrow \plim{k} (\dot \B')_k \hookrightarrow
\plim{k}(\D'_{\Linfty})_k
\]
follow immediately from $\D'_{\Lp} \hookrightarrow \dot \B' \hookrightarrow \D'_{\Linfty}$.

Conversely, let $S\in \plim{k}(\D'_{\Linfty})_k$, so that $(1+|x|^2)^kS \in \D'_{\Linfty}$ for all $k$. Then we also have, for all $k,l$ that
\[
(1+|x|^2)^kS \in (1+|x|^2)^{-l}\D'_{\Linfty}.
\]
Choosing now $l$ in such a way that $(1+|x|^2)^{-l} \in \D_{L^p}$ (i.e., $l>\frac{n}{2p}$), then
\cite[Thm.\ XXVI, p.\ 203]{Schwartz} shows that, for each $k$, $(1+|x|^2)^kS\in \D'_{L^p}$, i.e., $S\in \plim{k}(\D'_{\Lp})_k$. Finally, continuity of  
$\plim{k}(\D'_{\Linfty})_k \hookrightarrow \plim{k}(\D'_{\Lp})_k$ follows by 
transposing the continuous (by the closed graph theorem) embedding
\[
\ilim{k} (\D_{L^q})_{-k} \hookrightarrow \ilim{k} (\D_{L^1})_{-k}, 
\quad \frac{1}{p} + \frac{1}{q} = 1.
\]

By taking into account that the spaces 
$\plim{k}(\D'_{L^p})_k$ and $\plim{k}(\dot{\mathcal{B}}')_k$ are reduced projective limits and 
that bounded subsets of $\oc'$ are weakly relatively compact, the application of \cite[Ch.\ IV, 4.4, p.\ 139]{Schaefer} yields that the strong dual $(\oc')_b'$ of $\oc'$ is the space of uniformly slowly increasing
$\cinfty$-functions, i.e., 
\begin{equation}\label{eq:oc_def}
\oc:= (\oc')'_b = \ilim{k} \dot\B_{-k} = \ilim{k} (\D_{\Lp})_{-k}. 
\end{equation}
This also uses $(\dot\B')_b' \cong \D_{L^1}$, which in turn follows, e.g., from the fact that $\dot\B' \cong c_0\hato s'$ implies (cf.\ \cite[Thm.\ 1, p.\ 766]{Val13a})
\[
(\dot\B')_b' \cong \ell^1\hato  s \cong \D_{L^1}.
\]
The term ``uniform'' refers to the fact that for a function in $\oc$ all derivatives have the same polynomial growth at infinity.
Note that the space $\oc$ is \emph{ultrabornological}, being an inductive limit of Fr\'echet spaces. Furthermore, $\oc$ is the 
non-strict inductive limit of the Fr\'echet spaces $\dot\B_{-k}$ (cf.\ \cite[p.\ 173]{Horv}). The space $\om$ consists of all slowly increasing $\cinfty$-functions, i.e. (\cite[p.\ 243]{Schwartz}),
\[
\om = \plim{m} (\dot{\mathcal{B}}^m)_{-\infty} = \plim{m} (W^{m,p})_{-\infty} \qquad (1\le p \le \infty).
\]
Due to the characterization of functions $f\in \om$ by
\[
\forall \alpha \in \N_0^n \ \forall \vphi \in \Sch \colon \vphi \partial^\alpha f \in \cont_0\,, 
\]
we see that $\om$ is a closed subspace of the (uncountable) projective limit of the Banach spaces $(\partial^\alpha)^{-1}(\frac{1}{\vphi}) \cont_0$, $\vphi\in \Sch_+$, $\alpha\in \N_0^\alpha$, hence is complete. % (\cite[p.\ 52, 5.3]{Schaefer}). 
Note that, e.g., $e^{ix^2}\in \om(\R)\setminus \oc(\R)$. 

The Fourier transform $\Four: \oc' \to \om$ and its transpose $\Four: \om' \to \oc$ are isomorphisms (\cite[Th.\ XV, p.\ 268]{Schwartz}). The 
strong dual $\om'$ of $\om$ is the space of very rapidly decreasing distributions: If $S\in \om' = \Four\mo(\oc)$ then $\Four S \in \oc = 
\ilim{k} (\D_{\Lt})_{-k}$ and $S\in \ilim{k}(1-\Delta_n)^k(\Lt)_\infty$, i.e., 
\[
\om' = \ilim{k} (1-\Delta_n)^k (\Lt)_\infty \cong \ilim{m}(W^{-2m,2})_\infty.
\]
$\om'$ is ultrabornological and $\oc'$ is complete. This leaves the question of whether the strong dual $(\oc)'_b$ of the space $\oc$
is isomorphic to $\oc'$, i.e., $(\oc)'_b \cong \oc'$, which amounts to proving the reflexivity of $\oc$.

More generally, L.\ Schwartz posed the question whether the spaces $\om$ and $\oc'$ have Properties of topological vector spaces analogous 
to those of $\D$ and $\D'$ (\cite[p.\ 245]{Schwartz}). By this, he means: Are the spaces $\oc', \om, \oc, \om'$ ultrabornological complete Montel spaces? 
A.\ Grothendieck answered L.\ Schwartz' question in \cite[II, Th.\ 16, p.\ 131]{Groth_tensor}, supplemented by the property of nuclearity of the four 
spaces mentioned above.

The most difficult question here seems to be that of the completeness of $\oc \cong \om'$: ``Il n'est pas trivial, par contre que $\om'$
soit complet, ce que nous allons pourtant montrer'' (\cite[II, p.\ 130]{Groth_tensor}). The completeness of $\oc$ is the completeness of the non-strict $\mathcal{LF}$-space $(\dot{\mathcal{B}})_{-\infty} = \ilim{k} (\dot{\mathcal{B}})_{-k}$.

A.\ Grothendieck's proof of the ultrabornologicity of the space $\om$ (and thereby the completeness of its dual $\om'$, isomorphic to $\oc$) 
consists of two steps:
\begin{itemize}
\item[(i)] $\om$ is isomorphic to a direct factor of 
\[
s\hato s' \cong \Sch \hato \Sch' \cong \mathcal{L}_b(\Sch,\Sch)\,;
\]
(\cite[II, Lem.\ 18, p.\ 132--134]{Groth_tensor})
\item[(ii)] $s\hato s'$ is ultrabornological.
\end{itemize}
The main point is (ii) (see \cite[II, Lem.\ 18, p.\ 125--128]{Groth_tensor}). 
The first assertion (i) led M.\ Valdivia to proving the isomorphism
\begin{equation}\label{eq:V}
\om \cong s\hato s'
\end{equation}
(\cite[Th.\ 3, p.\ 478]{Val13}). Later on, Ch.\ Bargetz showed that $\oc \cong s \hato_\iota s'$ by using the isomorphism \eqref{eq:V} and the 
ultrabornologicity of $\om$ (\cite[Prop.\ 1, p.\ 318]{Bargetz}). This isomorphism immediately implies the completeness of the space $\oc$. However we don't use Ch.\ Bargetz' isomorphism because its proof relies on the ultrabornologicity of $s'\hato s$.
% We will give new simple proofs of the above sequence space representations in Section \ref{sec:5}. The term `simple' refers to the use of the 
% explicit sequence space representation
% \[
% \dltmk \cong s \hato \ltmk
% \]
% derived in Section \ref{sec:3} as well as to the proof of the completeness of the inductive limit $\ilim{k}(s\hato \ltmk)$ in Section \ref{sec:4}. 
\subsection{Objective}
The aim of our study is twofold: First, we give a proof of the compact regularity of the $\mathcal{LF}$ sequence space
\[
s \hato_\iota s' = \ilim{k} (s\hato (\ell^p)_{-k}) \cong \ilim{k} (s\hato (c_0)_{-k}) \qquad (1\le p\le \infty)
\]
(Proposition \ref{prop:compactly_regular}) as an illustration of Thm.\ 6.4 in \cite{Wengenroth14} and of Thm.\ 5.18 in \cite[p.\ 82]{vogt} (Section \ref{sec:4}).

Second, we derive the completeness of $\oc$ and the ultrabornologicity of $\om$ in Section \ref{sec:5}.

In Section \ref{sec:7} we prove that the strong dual of a quasinormable Fr\'echet space is a compactly regular
$\mathcal{LB}$-space. This is illustrated by several examples. Finally, the completeness result from Section \ref{sec:4}, i.e.,
\[
\ilim{k} (s\hato (\ell^\infty)_{-k}) = s \hato (\ilim{k} (\ell^\infty)_{-k}) = s\hato_\iota s' 
\]
is slightly generalized by Proposition \ref{Prop:new}. We also use the general results of A.\ Grothendieck on the strong duals of tensor products, as well as those 
on the permanence of properties in building completed tensor products. Also, L.\ Schwartz' theorem on the ultrabornologicity of strong duals of complete Schwartz spaces (\cite[I, p.\ 43]{Schwartz_vector}) is applied.

\subsection{Notation and conventions}\label{subsec:notation}
Our notation mostly follows \cite{Schwartz}, but for the reader's convenience we collect here all the basic definitions.
We denote by $\Delta_n:= \pa_1^2+\dots +\pa_n^2$ the Laplacean in $n$ variables, and by $Y$ the Heaviside function (\cite[p.\ 36]{Schwartz}).
$\cont_0$ is the space of continuous functions vanishing at infinity, while $\dot \B$ and $\D_{\Lp}$, $1\le p\le \infty$ are the complex-valued
$\cinfty$-functions defined on $\R^n$ all of whose derivatives belong to $\cont_0$ and $\Lp$, respectively (\cite[p.\ 199]{Schwartz}).
Subscripts $k$ refer to weights of the form $(1+|x|^2)^k$, e.g., $(\Lp)_{-k} = (1+|x|^2)^k \Lp$, $\dlpmk = (1+|x|^2)^k \dlp$,
or 
\[
\lpmk = j^k\ell^p  = \{ (x_j)_{j\in \N} \in \C^\N \colon \sum_{j\in \N} j^{-kp} |x_j|^p <\infty \},
\]
with the obvious modifications if $p=\infty$ or the index set is $\Z$ instead of $\N$.

The distribution spaces $\dlp' =(\dlq)'_b$, $1\le q<\infty$, $1/p+1/q=1$, $\dlo' = (\dot B)'_b$ are defined in \cite[p.\ 200]{Schwartz}. The dual of the Sobolev space $W^{2m, 2}$ (on $\R^n$) is the space $W^{-2m,2}:=(1-\Delta_n)^m\Lt$. The closure of the space $\mathcal{E}'$ of compactly supported
distributions in $\dli'$ is $\dot\B'$ (in particular, the latter is \emph{not} the dual space of $\dot \B$), see \cite[p.\ 200]{Schwartz}.

If $E$ is a locally convex Hausdorff topological vector space and if $E'$ is its dual, then $E'_b$ denotes $E'$ equipped with the strong
topology, i.e., the topology of uniform convergence on bounded subsets of $E$. By $s$ and $s'$ we denote the spaces of rapidly decreasing
and slowly increasing sequences, respectively. $\Sch$ and $\Sch'$ are the spaces of rapidly decreasing $\cinfty$-functions and of temperate 
distributions, respectively, as defined in \cite[p.\ 233--238]{Schwartz}. The Fourier transformation $\Sch \to \Sch$ and $\Sch'\to \Sch'$ is
uniformly denoted by $\mathcal{F}$.

A locally convex vector space $E$ is \emph{quasinormable} if the trace of the strong topology of $E_b'$ on each equicontinuous subset of $E'$
coincides with the topology of uniform convergence on a suitable neighborhood of $0$ of $E$ (\cite[4.\ Def., p.\ 98]{Bierstedt}, \cite[Def.\ 4, p.\ 106]{Groth}). A \emph{Schwartz space} is a quasinormable space whose bounded sets are precompact (\cite[Prop.\ 17, p.\ 116, Def.\ 5, p.\ 117]{Groth}). An inductive limit $\ilim{k} E_k$ is called \emph{regular} if each bounded subset $B$ of $\ilim{k} E_k$ is contained in
a step $E_k$ and is bounded there (\cite[p.\ 46]{Bierstedt}). $\ilim{k} E_k$ is called \emph{compactly regular} if each compact subset
$K$ of $\ilim{k} E_k$ is contained in a step $E_k$ and is compact there (\cite[p.\ 99/100]{Bierstedt}). Let $E$ and $F$ be locally convex spaces.
The \emph{inductive topology} on $E\otimes F$, denoted by $E\otimes_\iota F$, is the finest locally convex topology on $E\otimes F$ such that
the canonical map $E\times F \to E\otimes F$ is partially continuous (\cite[II, p.\ 13]{Schwartz_vector}, \cite[I, Def.\ 3, p.\ 74]{Groth_tensor}).
$E \hato_\iota F$ denotes the completion of $E\otimes_\iota F$. If one of the spaces $E$, $F$ is nuclear, then $E \hato F$ denotes the completed
tensor product of $E\otimes F$ with the $\eps$- or $\pi$-topology. For Fr\'echet and $\mathcal{DF}$-spaces $E, F$ we have $E\hato_\iota F = E \hato_\pi F$ (the latter denoting the 
projective tensor product). The $\beta$-topology on $E\otimes F$ is the finest locally convex topology on $E\otimes F$ such that the canonical map $E\times F \to E\otimes F$ is hypocontinuous (with respect to all bounded
subsets of $E$ and $F$), cf.\ \cite[II, p.\ 12]{Schwartz_vector}. $E\hato_\beta F$ denotes the completion of
$E\otimes_\beta F$. The index $\beta$ in $E\otimes_\beta F$ refers to the finest locally convex topology on $E\otimes F$ such that the canonical mapping $E\times F \to E\otimes F$ is hypocontinuous.

For $\mathcal{LF}$-spaces we adopt the notation from \cite[p.\ 44]{Bierstedt}. The index set is always $\Z$.

\section{Compact regularity and completeness of the \texorpdfstring{$\mathcal{LF}$-spaces}{LF-sp} \\ 
\texorpdfstring{$\mathbf{\ilim{k}(s\hato (\ell^p)_{-k}) 
\cong \ilim{k}(s\hato (c_0)_{-k}),\ 1\le p\le \infty}$}{limkslinfty}}\label{sec:4}

The compact regularity in question will be shown through an application of Theorem 2.7, (4)$\Leftrightarrow$(6) of \cite[p.\ 252]{Wengenroth14} via proving the equivalent property (Q) (a notation reminiscent of the concept of quasinormability) as defined in \cite[Prop.\ 2.3, p.\ 250]{Wengenroth14}: 
\begin{Definition}\label{def:property_Q}
Let $E = \ilim{k} E_k$ be an $\mathcal{LF}$-space (in particular, $E_k\hookrightarrow E_l \hookrightarrow E_m$ continuously whenever $k<l<m$). Then $E$ is said to have property (Q) if, for any $k\in \N$, there exists an
absolutely convex $0$-neighborhood $U_k$ in $E_k$ and an $l>k$ such that $E_l$ and $E_m$ induce the same topology on $U_k$
for all $m>l$.
\end{Definition}

\begin{Proposition}\label{prop:compactly_regular}
The inductive limits $\ilim{k} (s\hato (\ell^p)_{-k})$ and
$\ilim{k} (s\hato (c_0)_{-k})$, $1\le p\le \infty$
of the Fr\'echet spaces $s\hato (\ell^p)_{-k}$ and $s\hato (c_0)_{-k}$, respectively, are compactly regular and complete.
\end{Proposition}
\begin{proof} For the sake of simplicity, we restrict ourselves to the case of $p=\infty$. The other cases require $\ell^p$-norms instead of the $\ell^\infty$-norm.

Set $E_k := s\hato \limk$, so that
\[
E_k = \{(x_{ij})_{i,j} \in \C^{\N\times \N} \colon \forall r\in \N_0\ \sup_{i,j} i^r j^{-k} |x_{ij}| < \infty  \}
\]
and the topology of $E_k$ is induced by the seminorms
\[
p_{kr}((x_{ij})_{ij}) = \| (i^rj^{-k} x_{ij})_{i,j} \|_\infty.
\]
Condition (Q) for the three (continuously embedded) spaces $E_k \hookrightarrow E_l \hookrightarrow E_m$, $k< l < m$ means that there exists 
an absolutely convex neighborhood $U_k$ of $0$ in $E_k$ on which, for any such choice of $l, m$, $E_l$ and $E_m$ induce the same topology on $U_k$.
To see this it suffices to show that a sequence $(x_{ij}^\nu)$ in $U_k$ converging to $0$ in $E_m$ as $\nu\to \infty$ also does so in $E_l$.

The topologies of $E_l$ and $E_m$ are defined by the seminorms $p_{l,s}$, $s\in \N_0$ and $p_{m,t}$, $t\in \N_0$, respectively.
Note that $p_{k,r}$ is not defined on $E_l$, and $p_{l,s}$ is not defined on $E_m$.

Let $U_k := \{(x_{ij})_{i,j} \in \C^{\N\times \N} \mid p_{kr}((x_{ij})_{i,j})<1\}$ for an arbitrary $r\in \N$ and let $f_\nu = (x_{ij}^\nu)_{i,j\in \N}$ be a sequence in $U_k$ that converges to $0$ in $E_m$. Then we have:
\begin{align}
&\forall \nu,\, i,\, j\in \N \quad i^rj^{-k} |x^\nu_{ij}| < 1 \label{eq:x-bounded}\\
&\forall t\in \R^+ \quad \sup_{i,j} i^t j^{-m} |x^\nu_{ij}| \to 0 \quad (\nu\to \infty)
\label{eq:x-convergent}
\end{align}
Given any $s\in \N$ with $s>r$, we need to conclude from \eqref{eq:x-bounded}, \eqref{eq:x-convergent} that
$p_{ls}(f_\nu) \to 0$ as $\nu\to \infty$ (the requirement $s>r$ is no restriction since the seminorms in
$E_l$ are increasing in $s$). We will do this via interpolation. Let 
\[
\theta := \frac{m-l}{m-k} \in (0,1), \qquad t:= \frac{s-\theta r}{1-\theta} \in \R^+.
\]
Then $-l = \theta (-k) + (1-\theta)(-m)$ and $s = \theta r + (1-\theta) t$, so we obtain for any $i,j\in \N$:
\begin{align*}
i^sj^{-l}|x_{ij}^\nu| &= i^{\theta r + (1-\theta)t} j^{\theta (-k) + (1-\theta)(-m)} |x_{ij}^\nu|\\ 
&= (i^r j^{-k} |x_{ij}^\nu|)^\theta (i^tj^{-m}|x_{ij}^\nu|)^{1-\theta}
\stackrel{\eqref{eq:x-bounded}}{\le} (\sup_{i,j} i^t j^{-m} |x^\nu_{ij}|)^{1-\theta}.
\end{align*}
Together with \eqref{eq:x-convergent} this implies that $p_{ls}(f_\nu)\to 0$ as $\nu\to \infty$, establishing property (Q) and thereby compact regularity of $\ilim{k} (s\hato \limk)$ (using \cite[Th.\ 2.7, p.\ 252]{Wengenroth14}).

Finally, completeness of $\ilim{k} (s\hato \limk)$ follows from \cite[Cor.\ 2.8, p.\ 252]{Wengenroth14}.
\end{proof}
\begin{remark} Essentially the same proof can be given by applying Thm.\ 5.18 in \cite[p.\ 82]{vogt} to the $\mathcal{LF}$-space
\[
\ilim{k}(s\hato (\ell^\infty)_{-k}) = \{(x_{ij})_{i,j\in \N} \colon \exists k\in \N, \ \forall r\in \N\colon \sup_{i,j} i^r j^{-k} |x_{ij}| <\infty\}
\]
of double sequences.
\end{remark}
\begin{Corollary}\label{cor:2.4} For $1\le p \le \infty$,
\[
s\hato_\iota s' \cong \ilim{k}(s\hato (\ell^p)_{-k}) \cong \ilim{k}(s\hato (c_0)_{-k}).
\]
\end{Corollary}
\begin{proof}
Apply \cite[Rem.\ 22, p.\ 76]{Bargetz_thesis} or \cite[I, Prop.\ 14, p.\ 76]{Groth_tensor}.
\end{proof}
More generally, we have the following generalization of \cite[Prop.\ 3, p.\ 320]{Bargetz}:
\begin{Proposition}
Let $E$ and $F_k$, $k\in \N$, be Fr\'echet spaces, $F_k'$ the strong dual of $F_k$. If the inductive limit $\ilim{k}(E\hato_\iota F_k')$ is complete, then 
\[
E\hato_\iota (\ilim{k} F_k') \cong \ilim{k}(E\hato_\iota F_k').
\]
\end{Proposition}
\begin{proof}
Identical to the proof of Corollary \ref{cor:2.4}. 
\end{proof}
\section{Completeness of the space \texorpdfstring{$\mathbf{\oc}$}{OC} and ultrabornologicity of  \texorpdfstring{$\om$}{OM}}\label{sec:5} 
\begin{Proposition}\label{prop:3.1} (\cite[II, Th.\ 16, p.\ 131]{Groth_tensor}, \cite[Prop.\ 2]{Larcher_Wengenroth}, \cite[Ex.\ 6.7, p.\ 857]{Debrouwere_Vindas} ) 
The space $\om$ of slowly increasing functions is ultrabornological.
\end{Proposition}
\begin{proof}
(a) Corollary \ref{cor:2.4} and 1.\ Prop.\ in \cite[p.\ 57]{Bierstedt} imply that
\[
(s\hato_\iota s')'_b \cong (\ilim{k} (s\hato (\ell^q)_{-k}))'_b \cong \plim{k} (s\hato (\ell^q)_{-k})'_b.
\]
Furthermore, by Thm.\ 12 in \cite[II, p.\ 76]{Groth_tensor} we have:
\[
(s\hato (\ell^q)_{-k})'_b \cong s' \hato (\ell^p)_{k}
%(s\hato (\ell^p)_{k})_b' \cong s' \hato (\ell^q)_{-k}, \quad 1\le p < \infty, \ \frac{1}{p} + \frac{1}{q} = 1.
\]
Consequently, $(s\hato_\iota s')'_b \cong \plim{k} (s'\hato (\ell^p)_k)$, and the limit is interchangeable with the completed $\pi$-tensor product by \cite[p.\ 390]{Jarchow}. Therefore, we obtain $(s\hato_\iota s')'_b \cong s'\hato s$.

(b) The space $s\hato_\iota s'$ is a complete Schwartz space by Thm.\ 9-4 in \cite[II, p.\ 47]{Groth_tensor}. Therefore
by (a) and by L.\ Schwartz' Thm.\ on the ultrabornologicity of the strong dual of a complete Schwartz space \cite[I, p.\ 43]{Schwartz_vector} we conclude that $s'\hato s$ is ultrabornological.

(c) Due to M.\ Valdivia's isomorphism \cite[Thm.\ 3, p.\ 478]{Val13} $s'\hato s \cong \om$
we obtain from (b) the ultrabornologicity of $\om$.
\end{proof}

\begin{Proposition} (\cite[II, Th.\ 16, p.\ 131]{Groth_tensor}) 
The space $\oc$ of uniformly slowly increasing functions is complete.
\end{Proposition}
\begin{proof}
The ulrabornologicity of $\om$ implies the completeness of the space $\om'$, the space of very rapidly decaying destributions. Hence, also the isomorphic space $\oc$ is complete.
\end{proof}

The observations in the proof of Proposition \ref{prop:3.1} leads to the following generalization of Corollary \ref{cor:2.4} and, in particular, to a representation of the strong dual of the $\iota$-tensor product of an $\mathcal{LF}$-space with an $\mathcal{F}$-space.
\begin{Proposition}\label{Prop:new}
Let $E = \ilim{k} E_k$ be an $\mathcal{LF}$-space, $E_k$ and $F$ Fr\'echet spaces. If $\ilim{k} (E_k\hato_\pi F)$ is complete, then we obtain
\begin{itemize}
\item[(i)] 
\[
E\hato_\beta F = E\hato_\iota F = (\ilim{k} E_k)\hato_\iota F = \ilim{k}(E_k\hato_\pi F).
\]
\item[(ii)] If, furthermore, the space $E$ is complete and the space $F$ is nuclear,
then
\[
(E\hato_\iota F)'_b = E_b' \hato F_b'.
\]
\item[(iii)] Finally, if also $E$ and $F$ are nuclear and $E$ is a complete $\mathcal{DF}$-space, then $E'_b\hato F_b'$ is ultrabornological.
\end{itemize}
\end{Proposition}
\begin{proof}
(i) is a consequence of \cite[Rem.\ 22, p.\ 76]{Bargetz_thesis} or \cite[I, Prop. 14, p.\ 76]{Groth_tensor}.

(ii) By 1.\ Proposition in \cite[p.\ 57]{Bierstedt} and Grothendieck's duality theorem \cite[II, Th.\ 12, p.\ 76]{Groth_tensor} we obtain that
\begin{align*}
(E\hato_\iota F)'_b = (\ilim{k}(E_k\hato F)))_b' = \plim{k}(E_k\hato F)_b' = \plim{k}((E_k)'_b \hato F_b') = E_b' \hato F_b'.
\end{align*}

(iii) $E\hato_\iota F$ is nuclear by \cite[II, Th.\ 9-4, p.\ 47]{Groth_tensor} and hence is a complete Schwartz space. Its dual $E_b'\hato F_b'$ is then ultrabornological due to \cite[I, p.\ 43]{Schwartz_vector}.
\end{proof}
Corollary \ref{cor:2.4} then corresponds to the special case $E_k=(\ell^2)_{-k}$, $F=s$, $E=s'$ in Proposition \ref{Prop:new}.

\section{Strong duals of quasinormable Fr\'echet spaces}\label{sec:7}
In \cite[Prop.\ 2.2, p.\ 1696]{Bargetz-Nigsch-Ortner} the compact regularity of the space $\dlp' = \ilim{m} (1-\Delta_n)^m L^p$ was proved by direct methods. A second proof relying on the sequence space representation $\dlp' \cong s'\hato \ell^p$ was given in \cite[Prop.\ 4.3(b), p.\ 1700]{Bargetz-Nigsch-Ortner}.

The following result provides a simple criterion for establishing the compact regularity of $\mathcal{LB}$-spaces generated by duals of Fr\'echet spaces.
\begin{Proposition}\label{prop:strong_dual}
The strong dual of a quasinormable Fr\'echet space is a complete and compactly regular $\mathcal{LB}$-space.
\end{Proposition}
\begin{proof}
Let $\{U_k \mid k\in \N\}$ be a basis of closed, absolutely convex neighborhoods of $0$ with $U_k\supset U_{k+1}$ of the Fr\'echet space $F$. Quasinormability of $F$ implies (\cite[p.\ 99]{Bierstedt}) that the strong dual $F_b'$ of $F$ coincides with the $\mathcal{LB}$-space $\ilim{k}F_k'$, where $F_k':= \langle U_k^\circ \rangle$ is the Banach space generated by the polar $U_k^\circ$ of $U_k$. 

Denote by $q_k$ the Minkowski functional of $U_k^\circ$. Then the quasinormability of $F$ implies that for each $k\in \N$ there exists some $l>k$ such that for each $m>l$ and each $\eps>0$ there exists $C(\eps)>0$ such that the following (Q)-inequality is valid (\cite[26.14 Lem., p.\ 298]{Meise-Vogt}):
\begin{equation}\label{eq:Q}
q_l(f') \le \eps q_k(f') + C(\eps) q_m(f'), \qquad f'\in F_k' \hookrightarrow F_l' \hookrightarrow F_m'
\end{equation}
This inequality immediately implies that $F_b' = \ilim{k} F_k'$ satisfies property (Q) (see Definition \ref{def:property_Q}), hence
\cite[Th.\ 2.7, p.\ 252]{Wengenroth14} implies its compact regularity and \cite[Cor.\ 2.8, p.\ 252]{Wengenroth14} its completeness.
\end{proof}
\begin{remark}
As a referee remarked, Proposition \ref{prop:strong_dual} can also be found as Thm.\ (4)
in \cite[p.\ 301,302]{Bonet}. 
\end{remark}

\begin{Examples} \rm
The $\mathcal{LB}$-spaces
\[
s' = \ilim{k} (\ell^2)_{-k},
\ \mathcal{S}' = \lim_{\substack{k \rightarrow \\ m\rightarrow}} (1-\Delta_n)^m(\mathcal{C}_0)_{-k}, \ 
\mathcal{E}'(\Omega), \ \dlp' = \ilim{m} (1-\Delta_n)^m L^p
\]
($\Omega$ open in $\R^n$, $1\le p\le \infty$) are compactly regular on account of the quasinormability 
of the spaces $s, \ \mathcal{S},\ \mathcal{E}(\Omega), \ \mathcal{D}_{L^q} \cong s\hato \ell^q$, $1\le q <\infty$,
$\dot{\mathcal{B}}\cong s\hato c_0$ (use Grothendieck's permanence Proposition 13 in \cite[II,p.\ 76]{Groth_tensor}).

If $\mathcal{H}(\Omega)$ denotes the space of holomorphic functions on $\Omega$, then its strong dual $\mathcal{H}(\Omega)_b'$ is a 
compactly regular $\mathcal{LB}$-space due to the nuclearity and, hence, the quasinormability of the Fr\'echet space $\mathcal{H}(\Omega)$.
More generally, if $P(\partial)$ is a hypoelliptic differential operator with constant coefficients then $\mathrm{ker} P(\partial)$ (with the topologies induced by $\D'(\Omega)$ or $\mathcal{E}(\Omega)$)
is nuclear, hence quasinormable. Consequently, $(\mathrm{ker}P(\partial))_b'$ is compactly regular. $\mathcal{H}(\Omega)_b'$ then corresponds
to the special case $P(\partial) = \partial_1 + i\partial_2$, $\mathrm{ker}P(\partial) = \mathcal{H}(\Omega)$. 
\end{Examples}
Finally, we give a criterion for the validity of interchanging limits in completed tensor products:
\begin{Proposition}\label{prop:limit_interchange}
Let $\ilim{k}E_k$ be an $\mathcal{LF}$-space and let $F$ be a 
Fr\'echet space. If the inductive limit $\ilim{k} (E_k\hato_\iota F)$ is compactly regular, then
\[
\ilim{k} (E_k\hato_\iota F) = (\ilim{k}E_k) \hato_\iota F.
\]
\end{Proposition}
\begin{proof}
\cite[Rem.\ 22, p.\ 76]{Bargetz_thesis} or \cite[I, Prop.\, 14, p.\ 76]{Groth_tensor} yields for Hausdorff locally convex spaces $E_k$ and $F$:
\[
\ilim{k} (E_k\hato_\iota F) = (\ilim{k} E_k) \hato_\iota F
\]
if the limit $\ilim{k} (E_k\hato_\iota F)$ is complete. For Fr\'echet spaces $E_k$ and $F$ the compact regularity of
$\ilim{k} (E_k\hato_\iota F)$ implies its completeness by \cite[Cor.\ 2.8, p.\ 252]{Wengenroth14}.
\end{proof}
Note that Proposition \ref{prop:limit_interchange} does not imply point (a) in the proof of Proposition \ref{prop:3.1}: while compact regularity implies completeness, the converse is not true in general.
\begin{Examples} \
\begin{itemize}
\item[(a)] The $\mathcal{LB}$-space $\ilim{k}((\ell^\infty)_{-k}\hato c_0)$ is compactly regular due to the (Q)-inequality 
\begin{equation}\label{eq:exQ}
j^{-l}|x_{ij}| \le \eps j^{-k}|x_{ij}| + C(\eps) j^{-m} |x_{ij}|, \quad k<l<m,
\end{equation}
(for any $\eps>0$) if $(x_{ij})_{i,j} \in (\ell^\infty)_{-k}\hato_\pi c_0 \hookrightarrow (\ell)_{-l}\hato_\pi c_0 \hookrightarrow (\ell)_{-m}\hato_\pi c_0$, and due to \cite[6.4 Th., p.\ 112]{Wengenroth-derived}. 
To see \eqref{eq:exQ}, it suffices to note that for any $R>0$ and $k<l<m$ we have
\begin{align*}
j^{-l} &= Y(R-j)j^{-l} + Y(j-R)j^{-l} = Y(R-j)j^{-m}j^{m-l} + Y(j-R)j^{-k}j^{k-l}\\
&\le R^{m-l}j^{-m} + R^{k-l}j^{-k}
\end{align*}
Hence by Proposition \ref{prop:limit_interchange} the space
\[
\ilim{k}((\ell^\infty)_{-k}\hato_\pi c_0) \cong \ilim{k}((\ell^\infty)_{-k})\hato_\pi c_0 \cong s' \hato c_0 \cong
\dot{\mathcal{B}}'
\]
is compactly regular.

Note that the space $\dot{\mathcal{B}}'$ is also a $\mathcal{DF}$-space, being an inductive limit of Banach spaces (\cite[p.\ 64]{Groth}).
\item[(b)] Similarly, the compact regularity of the $\mathcal{LB}$-spaces $(L^p)_{-\infty}, (\mathcal{M}^1)_{-\infty}$ (\cite[p.\ 241]{Schwartz}), $\oc^\circ, \oc^m$ (\cite[p.\ 173]{Horv}) is an easy consequence of a simple (Q)-inequality, compare \cite[Prop.\ 5.2, p.\ 932]{Bargetz-Nigsch-Ortner2}, with respect to a projective limit description of these spaces. 
\end{itemize}
\end{Examples}

\medskip\noindent
{\bf Acknowledgements.} We are greatly indebted to one referee for pointing out a mistake in an earlier version of the paper and for several helpful suggestions.

\end{document}